\documentclass[reqno]{amsart}
\usepackage[latin1]{inputenc}
\usepackage{amsmath}
\usepackage{amsfonts}
\usepackage{amssymb}
\usepackage{fullpage}
\usepackage{hyperref}
\usepackage{enumitem}
\usepackage{todonotes}
\usepackage{cleveref}
\usepackage{breqn}
\setkeys{breqn}{breakdepth={1}}

\theoremstyle{plain}
\newtheorem{theorem}{Theorem}
\newtheorem{corollary}[theorem]{Corollary}
\newtheorem{lemma}[theorem]{Lemma}

\theoremstyle{definition}

\newtheorem{example}[theorem]{Example}

\theoremstyle{remark}

\begin{document}
    
    \author{Hieu D. Nguyen \and Long G. Cheong}
    \address{Department of Mathematics, Rowan University, Glassboro, NJ 08028, USA.}
    \email{nguyen@rowan.edu, cheong94@students.rowan.edu}
    \title{New Convolution Identities for \\ Hypergeometric Bernoulli Polynomials}
    \date{11-14-2013}
    
    %\thanks{}
    \begin{abstract}
        New convolution identities of hypergeometric Bernoulli polynomials are presented.  Two different approaches to proving these identities are discussed, corresponding to the two equivalent definitions of hypergeometric Bernoulli polynomials as Appell sequences.
        
    \end{abstract} 
    
    \maketitle

    \section{Introduction}
    It is well known that the Euler-Maclaurin Summation (EMS) formula given by
    \begin{equation}
    \sum\limits_{k = 0}^n f(k) = \int_0^n f(x)dx + \frac{1}{2} [f(n)+f(0)] + \sum\limits_{k = 2}^\infty \frac{B_{k}}{k!} \left[f^{(k - 1)}(n) - f^{(k - 1)}(0) \right]
    \label{eqn:euler-maclaurin_summation_formula}
    \end{equation}
is extremely useful for approximating sums and integrals and for deriving special formulas.  Here, $B_n$ are the Bernoulli numbers defined by the exponential generating function
$$
\frac{t}{e^t-1}=\sum_{n=0}^{\infty} B_n \frac{t^n}{n!}
$$
For example, if we set $f(x)=x^p$ in \eqref{eqn:euler-maclaurin_summation_formula} and use the fact that $B_0=1$ and $B_1=-1/2$, then we obtain the classical sums of powers formula first discovered by Jacob Bernoulli:
$$
    \sum\limits_{k = 1}^n {k^p} = n^p + \sum\limits_{k = 0}^p \frac{p!}{k!(p - k + 1)!} B_k n^{p + 1 - k}
$$
Consider next the special case of the EMS formula where $n=1$, which we shall write in the form
     \begin{equation}
    \int_0^1 f(x)dx = \frac{1}{2} [f(1)+f(0)] - \sum\limits_{k = 2}^\infty \frac{B_{k}}{k!} \left[f^{(k - 1)}(1) - f^{(k - 1)}(0) \right]
    \label{eq:em-special-case}
    \end{equation}
If we again set $f(x)=x^n$ in \eqref{eq:em-special-case}, then we obtain the classic Bernoulli number identity first discovered by Euler:
$$
    \sum\limits_{k = 0}^{n} \binom{n + 1}{k} B_k = 0
$$
    
    It is natural to ask if other Bernoulli number identities can be obtained by substitution.  For example, is there a function $f(x)$ which when substituted into \eqref{eq:em-special-case} will yield the following quadratic identity?
    \begin{equation}\label{eq:euler}
    \sum\limits_{k = 0}^{n + 1} \binom{n + 1}{k} B_k B_{n - k + 1} = -(n + 1)B_n - n B_{n + 1}
\end{equation}
    The answer, not surprisingly, is yes.  The surprise however is the choice for $f(x)$.  It is clear that $f(x)$ should involve the Bernoulli numbers since \eqref{eq:euler} contains products of Bernoulli numbers.  Therefore, a natural choice for $f(x)$ would be to set it equal to a Bernoulli polynomial, say $B_n(x)$.   However, the reader will discover that substituting $f(x)=B_n(x)$ into  \eqref{eq:em-special-case} yields the trivial identity.  The correct answer is $f(x)=(1-x)B_n(x)$.
    
    The Bernoulli polynomials $B_n(x)$ give an example of an Appell sequence.  As such, there are two equivalent definitions for $B_n(x)$: one via the exponential generating function
    \begin{equation}
    \label{eq:egf}
    \frac{te^{xt}}{e^t-1}=\sum_{n=0}^{\infty}B_n(x)\frac{t^n}{n!}
    \end{equation}
and the other as a polynomial sequence with the following properties:
    
        \begin{equation}
    \begin{aligned}
    B_0(x) &= 1,\\
    {B_n}'(x) &= n B_{n - 1}(x),\\
    \int_0^1 B_n(x)\,dx &= \delta_n \equiv \begin{cases}
    1 & \text{if } n = 0\\
    0 & \text{if } n \ne 0
    \end{cases}
    \end{aligned}
    \label{eq:appell-sequence}
    \end{equation}
where $\delta_n$ is the Kronecker delta function.  In either approach the Bernoulli numbers can be computed as the evaluation $B_n = B_n(0)$.   Beginning with the second approach, we shall demonstrate that this technique of substitution to obtain Bernoulli number identities can be extended to Appell sequences by generalizing the EMS formula.   This was achieved by recognizing that \eqref{eq:em-special-case}  can be derived from the following repeated integration by parts formula
    
    \begin{equation}
    \label{eq:integrationbyparts}
    \int{f(x)g(x)dx} = \sum_{k=0}^p (-1)^k f^{(k)}(x)g^{-(k+1)}(x) +(-1)^{-(p+1)}\int{f^{(p+1)}(x)g^{-(p+1)}(x)dx}
    \end{equation}
where $f^{(k)}(x)$ denotes the $k$-th derivative of $f(x)$ and $g^{-(k)}(x)$ denotes the $k$-th integral (or anti-derivative) of $g(x)$.  In particular, if we set $g(x)=B_0(x)$ in \eqref{eq:integrationbyparts}, then $g^{-(k)}(x)= B_k/k!$ because of the derivative property in \eqref{eq:appell-sequence}.  Then integrating over the interval $[0,1]$ and using the fact that $B_1=-1/2$, $B_{2k+1}=0$ and $B_k(1)=(-1)^k B_k(0)$ for $k\geq 1$, we find that  \eqref{eq:integrationbyparts} reduces to \eqref{eq:em-special-case} in the limit where $p \rightarrow \infty$.
    
    Let $q_n(x)$ be an Appell sequence, i.e. $q_0(x)=1$ and $q_n'(x)=nq_{n-1}'(x)$.  Then setting $g(x)=q_0(x)$ in \eqref{eq:integrationbyparts} yields the following generalized EMS formula:
\begin{equation}
\label{eq:ems-appell}
    \int{f(x)dx} = \sum_{k=0}^p (-1)^k f^{(k)}(x)g^{-(k+1)}(x) +(-1)^{-(p+1)}\int{f^{(p+1)}(x)g^{-(p+1)}(x)dx}
\end{equation}
More interestingly, if we also set $f(x)=p_n(x)$ to be another Appell sequence in \eqref{eq:ems-appell}, then we obtain the following convolution identity (Theorem 1) for any pair of Appell sequences $p_n(x)$ and $q_n(x)$:
\begin{equation}
\label{eq:convolution-appell-intro}
\sum_{k=0}^n (-1)^k \binom{n}{k} p_{n-k}(x)q_{k}(x) = c_n
\end{equation}
 where $\{c_n\}$ are constants independent of $x$.

Equation \eqref{eq:convolution-appell-intro} is our starting point for deriving new identities.  We shall prove \eqref{eq:convolution-appell-intro} in the next section and apply it to a class of generalized Bernoulli polynomials known as hypergeometric Bernoulli polynomials, defined by equation \eqref{eq:egf-hyper} (see also \cite{hn}), in order to obtain new convolution identities.  For example, the following sums of products formula holds for any two hypergeometric Bernoulli numbers $B_{M,n}$ and $B_{N,n}$ of order $M$ and
$N$, respectively, which we prove in Theorem \ref{th:hypergeometric-bernoulli}:

    \begin{eqnarray}
\lefteqn{\sum_{\substack {0\leq i \leq M-1 \\ 0 \leq j \leq N-1 \\(i,j)\neq (0,0)}} \left[\sum_{k=0}^n (-1)^k \binom{n}{j;i;k-i} B_{N,n-k-j} B_{M,k-i}\right]= }  \notag  \\  &&   (-1)^{M-1} \left[\binom{M+N}{M}\delta_{n-M-N}+\sum_{j=0}^{N-1}  \binom{n}{M;j}B_{N,n-M-j} +(-1)^{n-M-N}\sum_{i=0}^{M-1}  \binom{n}{N; i}B_{M,n-N-i}\right] \ \label{eq:main-result-1}
    \end{eqnarray}
where $\binom{a}{b; c}$ denotes  the multinomial coefficient defined by 
$$
\binom{a}{b; c} = \frac{a!}{b!c!(a-b-c)!}
$$
In the special case where $M=N=2$, we obtain the identities
$$
        \sum_{k=0}^{n} (-1)^k\binom{n}{k}B_{2,n-k}(x)B_{2,k}(x) =\frac{1}{2}\delta_{n-2}+nB_{2,n-1}+B_{2,n}
$$
and
$$
    \sum_{k=0}^{n} \binom{2n}{2k}B_{2,2n-2k}B_{2,2k} =\frac{1}{4}[\delta_{2n-2}+2nB_{2,2n-1}-(2n-4)B_{2,2n}]
$$
which we prove in Corollaries \ref{cor:N=M=2} and \ref{cor:N=M=2-even}, respectively.  The latter formula generalizes Euler's quadratic formula for the classical Bernoulli numbers:
    \[\sum\limits_{k = 0}^n \binom{2n}{2k} B_{2k} B_{2n - 2k} = -(2n - 1) B_{2n}\]

Lastly, in section 3 we demonstrate that these same identities can be derived by taking the other approach to hypergeometric Bernoulli polynomials, namely by considering their exponential generating functions, and employing special partial fraction expansion formulas given in Lemma \ref{le:partial-fraction}.  As a result, we obtain the following identities (Theorem \ref{th:HBP-egf}):
\begin{align} 
 \sum_{m=1}^{M+N-2} m! a_m \sum_{k=0}^{n-m} (-1)^k \binom{n}{m; k} B_{N,n-m-k}(x)B_{M,k}(x)
 = (-1)^M\binom{M+N}{M}\delta_{n-M-N} & \notag  \\ 
  \hspace{50pt} + (-1)^M\sum_{m=0}^{N-1} \binom{n}{m; M} B_{N,n-m-M} +
(-1)^{n-N}\sum_{m=0}^{M-1}  \binom{n}{m; N}  B_{M,n-m-N}  \label{eq:main-result-2a}
\end{align}
and if $N\geq M$,
\begin{flalign}
& \sum_{m=M}^{N-1} \sum_{k=0}^{n-m}  \binom{n}{m; k} B_{N,n-m-k}(x_1)B_{M,k}(x_2) 
= \binom{n}{M} B_{N,n-M}(z) - \binom{n}{N}B_{M,n-N}(z) \label{eq:main-result-2b}
\end{flalign}
where $z=x_1+x_2$.  Observe that \eqref{eq:main-result-2a} is equivalent to \eqref{eq:main-result-1}.  Moreover, \eqref{eq:main-result-2b} is trivial when $M=N$.  To remedy this, we employ a derivative expansion formula to establish a recurrence formula for sums of products of hypergeometric Bernoulli polynomials.  We then use this recurrence to obtain a general formula for these sums of products,  thereby generalizing Kamano's formula for hypergeometric Bernoulli numbers given in \cite{kamano} and Dilcher's formula for Bernoulli polynomials given in \cite{dilcher}.  An example is the identity
$$
\sum_{k=0}^n \binom{n}{k} B_{N,n-k}(x_1)B_{N,k}(x_2)=\frac{1}{N}\left[(N-n)B_{N,n}(z)+n(z-1)B_{N,n-1}(z)\right] 
$$

\section{Hypergeometric Bernoulli Polynomials}

Let $N$ be a positive integer.  Following the work of A. Hassen and the first author in \cite{hn} we define {\em hypergeometric Bernoulli polynomials} $B_{N,n}(x)$ of order $N$ by the exponential generating function
    \begin{equation}
    \label{eq:egf-hyper}
   F_N(x,t)\equiv \frac{(t^N/N!)e^{xt}}{e^t-T_{N-1}(t)}=\sum_{n=0}^{\infty}B_{N,n}(x)\frac{t^n}{n!}
    \end{equation}
Here, $T_{N}(t)=\sum_{n=0}^{N}\frac{t^n}{n!}$ is the $N$-th Taylor polynomial of $e^t$.   In particular, when $N=1$, we recover the classical Bernoulli polynomials, i.e. $B_{1,n}(x)=B_n(x)$.  Hypergeometric Bernoulli polynomials play an important role in the study of hypergeometric zeta functions (see \cite{hn2}).

It is known that hypergeometric Bernoulli polynomials can also be defined by the following properties:
        \begin{flalign}
       & \mathrm{(i)} B_{N,0}(x) = 1 \label{itm:hypergeometric-bernoulli-initial-value} & \\
       & \mathrm{(ii)} B_{N,n}'(x) = n B_{N,n - 1}(x) \label{itm:hypergeometricbernoullidifferentialdefinition} & \\
       & \mathrm{(iii)} \int_0^1 {(1 - x)}^{N - 1} B_{N,n}(x)\,dx = \frac{1}{N}\delta_n  &
        \end{flalign}  
These properties were shown to be equivalent to definition \eqref{eq:egf-hyper} in \cite{hn}.
In either definition, {\em hypergeometric Bernoulli numbers} are defined analogously by $B_{N,m}=B_{N,m}(0)$.  Moreover, we define $B_{N,n}(x)=0$ for $n<0$.

We begin with an alternating convolution formula valid for Appell sequences, i.e. polynomial sequences that satisfy conditions \eqref{itm:hypergeometric-bernoulli-initial-value} and \eqref{itm:hypergeometricbernoullidifferentialdefinition}.
        \begin{theorem}
\label{th:convolution-appell}
Let $p_n(x)$ and $q_n(x)$ be two Appell sequences.  Then
\begin{equation} \label{eq:convolution-appell}
\sum_{k=0}^n (-1)^k \binom{n}{k} p_{n-k}(x)q_{k}(x) = c_n
\end{equation}
        where $\{c_n\}$ are constants independent of $x$.  Moreover, if $p_n(x)=q_n(x)$, then $c_n=0$ for every odd positive integer $n$.
\end{theorem}

    \begin{proof}
        Set $ f(x) = p_{n-1}(x)$ , $g(x) = q_{0}(x)$, and $p = n-1$ in \eqref{eq:integrationbyparts}. Then using property \eqref{itm:hypergeometricbernoullidifferentialdefinition} for Appell sequences and the fact that $f^{(p+1)}=0$ (since $f$ is a polynomial of degree $p$), we obtain

        \begin{dmath*}
           \int p_{n-1}(x) = \sum_{k = 0}^{n-1} (-1)^k \frac{(n-1)_{(k)}}{(k+1)!} p_{n-1-k}(x) q_{k+1}(x)
        \end{dmath*}
where $n_{(k)}= n(n-1)(n-2)...(n-k+1)$ is the falling factorial.  On the other hand, since $p_{n}$ is an Appell sequence we have
        \begin{dmath*}
            \int p_{n-1}(x) = \frac{p_{n}(x)-c_n}{n} 
                    \end{dmath*}
        where $c_n$ is the constant of integration.  It follows from the substitution $ \ell = k +1 $ that
        
            \begin{dmath*}
                p_{n}(x) - c_n =  - \sum_{\ell = 1}^{n} (-1)^\ell \binom{n}{\ell} p_{n-\ell}(x) q_{\ell}(x)
            \end{dmath*}
        Rearranging terms, we get 
        \begin{dmath*}
            p_{n}(x) + \sum_{\ell = 1}^{n} (-1)^\ell \binom{n}{\ell} p_{n-\ell}(x) q_{\ell}(x) = c_n
        \end{dmath*}
        or equivalently,
        \[ \sum_{\ell = 0}^n (-1)^\ell \binom{n}{\ell}p_{n -\ell}(x) q_{\ell}(x) = c_n \]
        
        If we now assume $p_n(x)=q_n(x) $ and $ n $ is odd, then all the terms in the convolution sum cancel. Thus, $ c_n = 0$, which completes the proof.  
            \end{proof}
We note that Theorem \ref{th:convolution-appell} can also be proven by differentiating the left-hand side of \eqref{eq:convolution-appell} or by using exponential generating functions.  More precisely, define
\begin{equation}
C_n(x)=\sum_{k=0}^n (-1)^k \binom{n}{k} p_{n-k}(x)q_{k}(x) 
\end{equation}
Then using the fact that $p'_{n-k}(x)=(n-k)p_{n-k-1}(x)$ and $q'_k(x)=kq_{k-1}(x)$, we have
\begin{align}
C_n'(x) & = \sum_{k=0}^n (-1)^k \binom{n}{k} [(n-k)p_{n-k-1}(x)q_{k}(x)+kp_{n-k}(x)q_{k-1}(x)]  \notag \\
& =  n \sum_{k=0}^n (-1)^k \binom{n-1}{k}p_{n-k-1}(x)q_{k}(x)+n \sum_{k=0}^n (-1)^k \binom{n-1}{k-1}p_{n-k}(x)q_{k-1}(x)] \notag \\
& = n \sum_{k=0}^{n-1} (-1)^k \binom{n-1}{k}p_{n-k-1}(x)q_{k}(x)-n \sum_{k=0}^{n-1} (-1)^k \binom{n-1}{k}p_{n-k-1}(x)q_{k}(x)] \notag \\
& = 0 \notag
\end{align}
Thus, $C(x)$ is constant, which proves Theorem \ref{th:convolution-appell}.  Alternatively, $p_n(x)$ and $q_n(x)$ have exponential generating functions of the form
\begin{align}
F(x,t)= f(t)e^{xt}=\sum_{n=0}^{\infty}p_n(x)\frac{t^n}{n!} \notag \\
G(x,t)=g(t)e^{xt}=\sum_{n=0}^{\infty}q_n(x)\frac{t^n}{n!} \notag
\end{align}
It follows that
$$
F(x,t)G(x,-t)=\sum_{n=0}^{\infty}\left(\sum_{k=0}^n (-1)^k \binom{n}{k} p_{n-k}(x)q_k(x)\right)\frac{t^n}{n!}=\sum_{n=0}^{\infty}C_n(x)\frac{t^n}{n!}
$$
But we also have $F(x,t)=G(x,-t)=f(t)g(t)$, which is independent of $x$.  Thus, $C_n(x)=c_n$ is constant and Theorem  \ref{th:convolution-appell}  follows.
\

Let us now set $p_n(x)=B_{N,n}(x)$ and $q_n(x)=B_{M,n}$ in Theorem \ref{th:convolution-appell}.  This yields the following theorem.

    \begin{theorem}
        \label{th:convolution}
        Let $B_{M,n}(x)$ and $B_{N,n}(x)$ be two hypergeometric Bernoulli polynomials of orders $M$ and $N$, respectively.  Then
        \begin{equation} \label{eq:convolution}
        \sum_{k=0}^n (-1)^k \binom{n}{k} B_{N,n-k}(x)B_{M,k}(x) = c_n
        \end{equation}
        where $\{c_n\}$ are constants independent of $x$.  Moreover, if $M=N$, then $c_n=0$ for every odd positive integer $n$.
    \end{theorem}

Since the left hand side of \eqref{th:convolution} is independent of $x$, we obtain the following corollary as a result.

    \begin{corollary} For any two real values $a$ and $b$, we have
        \begin{equation} \label{eq:convolutiontwovalues}
       c_n=\sum_{k=0}^n (-1)^k \binom{n}{k} B_{N,n-k}(b)B_{M,k}(b)  = \sum_{k=0}^n (-1)^k \binom{n}{k} B_{N,n-k}(a)B_{M,k}(a)
        \end{equation}
    \end{corollary}

    Surprisingly, in the case of classical Bernoulli numbers, \eqref{eq:convolutiontwovalues} becomes trivial when $M=N=1$ and $a=0$, $b=1$ since it is known that $B_n(1)=(-1)^nB_n(0)$.  If we denote $B_n =B_n(0)$, then \eqref{eq:convolutiontwovalues} becomes
    \begin{equation} \label{eq:bernoulliconvolution}
    (-1)^n \sum_{k=0}^n (-1)^k \binom{n}{k} B_{n-k}B_k = \sum_{k=0}^n (-1)^k \binom{n}{k} B_{n-k}B_k
    \end{equation}
    When $n$ is odd, the left and right hand sums in \eqref{eq:bernoulliconvolution} are negatives of each other, which implies
    \begin{equation} \label{eq:bernoullisymmetricterms}
    \sum_{k=0}^n (-1)^k \binom{n}{k} B_{n-k}B_k=0
    \end{equation}
    However, we had already deduced this earlier since \eqref{eq:bernoullisymmetricterms} is just a restatement of \eqref{eq:convolution} with $c_n=0$ for $n$ odd.  On the other hand, when $n$ is even, then the left and right hand sides of \eqref{eq:bernoulliconvolution} are identical and we are left with a trivial identity.
    
    Fortunately, \eqref{eq:convolutiontwovalues}  can be used to obtain new non-trivial identities for hypergeometric Bernoulli numbers $B_{N,n}$ with $N>1$.  We demonstrate this next with the use of the following lemma which helps us to evaluate hypergeometric Bernoulli polynomials $B_{N,n}(x)$ at $x=1$.

    \begin{lemma} \label{le:formula-B_N}
        \begin{equation} \label{eq:formula-B_Nv2}
        B_{N,n+N}(1)=\frac{(n+1)^{(N)}}{N!}\delta_n + \sum_{j=1}^{N} \binom{n+N}{n+j} B_{N,n+j}
        \end{equation}
        or equivalently,

        \begin{equation} \label{eq:formula-B_N}
        B_{N,k}(1)=\frac{(k-N+1)^{(N)}}{N!}\delta_{k-N} + \sum_{j=0}^{N-1} \binom{k}{k-j} B_{N,k-j}
        \end{equation}
        where we define $B_{N,k}=0$ for $k<0$.
        
    \end{lemma}

\begin{proof}
Set $f(x)=(1-x)^{N-1}$, $g(x)=B_{N,n}(x)$, and $p=N-1$.  As $f$ and $g$ are Appell sequences, it follows from \eqref{eq:integrationbyparts} and the fact $f^{(p+1)}=0$ (since $f$ is a polynomial of degree $p$) that
$$
    \int{(1-x)^{N-1}B_{N,n}(x)dx} = \sum_{k=0}^{N-1} \frac{(N-1)_{(k)}}{(n+1)^{(k+1)}}(1-x)^{N-1-k} B_{N,n+k+1}(x)
$$
where $(N-1)_{(k)}$ and $(n+1)^{(k+1)}$ denote falling and rising factorials, respectively.
Integrating this equation over the interval $[0,1]$ yields
$$
    \int_0^1{(1-x)^{N-1}B_{N,n}(x)dx} =\frac{(N-1)!}{(n+1)^{N}} B_{N,n+N}(1)- \sum_{k=0}^{N-1} \frac{(N-1)_{(k)}}{(n+1)^{(k+1)}}B_{N,n+k+1}(0)
$$
Equating this answer with \eqref{itm:hypergeometricbernoullidifferentialdefinition} in the definition of hypergeometric Bernoulli polynomials and solving for $B_{N,n+N}$ gives \eqref{eq:formula-B_Nv2} as desired.  Formula \eqref{eq:formula-B_N} now follows easily from \eqref{eq:formula-B_Nv2} by making a change of variable and re-indexing.
\end{proof}

We now make use of Lemma \ref{le:formula-B_N} to prove identities involving sums of products of hypergeometric Bernoulli numbers.  

\begin{theorem} \label{th:hypergeometric-bernoulli}
    For positive integers $N$ and $M$, we have  
     \begin{multline}
\shoveleft{\sum_{\substack {0\leq i \leq M-1 \\ 0 \leq j \leq N-1 \\(i,j)\neq (0,0)}} \left[\sum_{k=0}^n (-1)^k \binom{n}{j;i;k-i} B_{N,n-k-j} B_{M,k-i}\right]}  \\
 =   (-1)^{M-1} \left[\binom{M+N}{M}\delta_{n-M-N}+\sum_{j=0}^{N-1}  \binom{n}{M;j}B_{N,n-M-j} +(-1)^{n-M-N}\sum_{i=0}^{M-1}  \binom{n}{N; i}B_{M,n-N-i}\right] \label{eq:hypergeometric-bernoulli}
    \end{multline}
where
$$ \binom{n}{n_1;n_2;...;n_k}=\frac{n!}{n_1!n_2!...n_k!(n-n_1-n_2-...-n_k)!}$$
denote a multinomial coefficient.
\end{theorem}

\begin{proof} We set $a=0$ and $b=1$ in \eqref{eq:convolutiontwovalues} and use \eqref{eq:formula-B_N} to obtain
        \begin{multline} 
        \label{eq:hypergeometric-expansion}
        \sum_{k=0}^n (-1)^k \binom{n}{k} \left[\frac{(n-k-N+1)^{(N)}}{N!}\delta_{n-k-N} + \sum_{j=0}^{N-1} \binom{n-k}{n-k-j} B_{N,n-k-j}(0)\right] \left[\frac{(k-M+1)^{(M)}}{M!}\delta_{k-M}\right. \\
        + \left .\sum_{i=0}^{M-1} \binom{k}{k-i} B_{M,k-i}(0)\right]
        = \sum_{k=0}^n (-1)^k  \binom{n}{k}  B_{N,n-k}(0)B_{M,k}(0)
        \end{multline}
        Then expanding the product inside the left-most summation above yields four separate summations that can be rearranged and simplified as follows:
        \begin{dmath}
        (-1)^M \left[\binom{M+N}{M}\delta_{n-M-N}+\sum_{j=0}^{N-1}  \binom{n}{M;j}B_{N,n-M-j} +(-1)^{n-M-N}\sum_{i=0}^{M-1}  \binom{n}{N;i}B_{M,n-N-i}\right]
+        \sum_{j=0}^{N-1} \sum_{i=0}^{M-1} \sum_{k=0}^n (-1)^k \binom{n}{j;i;k-i} B_{N,n-k-j} B_{M,k-i}
= \sum_{k=0}^n (-1)^k  \binom{n}{k}  B_{N,n-k}B_{M,k}
        \end{dmath}
Next, observe that the terms in the triple summation above for $i=j=0$ yield the same sum as that on the right hand side, thus cancelling each other.  Then solving for the remaining terms in the triple summation yields \eqref{eq:hypergeometric-bernoulli} as desired.
\end{proof}

We now consider special cases of Theorem \ref{th:hypergeometric-bernoulli}.  For example, the following corollary holds if $M=1$.

\begin{corollary} \label{cor:HBP-M=1}
For any positive integer $N$ we have
    \begin{equation}
\sum_{j=1}^{N-1}\sum_{k=0}^n  (-1)^k \binom{n}{j, \ k} B_kB_{N,n-k-j} = (N+1)\delta_{n-N-1} +(-1)^{n-N-1} \binom{n}{N}B_{n-N} + \sum_{j=0}^{N-1} \binom{n}{1, \ j} B_{N,n-1-j} 
    \label{eqn:hypergeometric_formulaM=1}
    \end{equation}
\end{corollary}

In addition, if we set $N=2$ in Corollary \ref{cor:HBP-M=1}, then we obtain

\begin{corollary}
For $n\geq 1$ we have
\begin{equation}
\sum_{k=0}^{n}  (-1)^k \binom{n}{k} B_kB_{2,n-k} =  B_{2,n}+ n B_{2,n-1}  -  \frac{n}{2}B_{n-1}
    \label{eqn:hypergeometric_formulaM=1andN=2}
    \end{equation}
or equivalently,
    \begin{equation}
    \sum\limits_{k = 0}^n \binom{n}{k} B_k B_{2,n - k} = B_{2,n} - \frac{n}{2} B_{n - 1}
    \label{eqn:hypergeometric_formulaM=1andN=2version2}
    \end{equation}
\end{corollary}

\begin{proof} Observe that if $N=2$ in \eqref{eqn:hypergeometric_formulaM=1}, then $j$ only takes on the value 1 for the left-hand side, in which case $ \binom{n}{1, \ k} =n \binom{n-1}{k}$.  Moreover, on the right-hand side we have $\binom{n}{1, \ j} = n\binom{n-1}{j}$.  This simplifies \eqref{eqn:hypergeometric_formulaM=1andN=2} as follows:
$$
\sum_{k=0}^n  (-1)^k n \binom{n-1}{k} B_kB_{2,n-k-1} = 3\delta_{n-3} + (-1)^{n-3} \binom{n}{2}B_{n-2}  + \sum_{j=0}^{1}  n \binom{n-1}{j} B_{2,n-1-j} 
$$
Next, we replace $n$ by $n+1$ and simplify to obtain
$$
(n+1) \sum_{k=0}^{n+1}  (-1)^k \binom{n}{k} B_kB_{2,n-k} = 3\delta_{n-2} +(-1)^{n-2} \frac{(n+1)n}{2}B_{n-1}  +(n+1)B_{2,n}  + (n+1)nB_{2,n-1}
$$
We then divide both sides by $n+1$ and use the fact that $3\delta_{n-2}/(n+1)=\delta_{n-2}$ on the right-hand side and $B_{n+1}=0$ on the left-hand side for every even integer $n> 1$ to obtain
$$
\sum_{k=0}^{n}  (-1)^k \binom{n}{k} B_kB_{2,n-k} = \delta_{n-2} + (-1)^{n-2} \frac{n}{2}B_{n-1} +B_{2,n}  + nB_{2,n-1}
$$
Equation \eqref{eqn:hypergeometric_formulaM=1andN=2} now follows from the fact that $\delta_{n-2}+(-1)^{n-2}\frac{n}{2}B_{n-1}= - \frac{n}{2}B_{n-1}$.  To obtain \eqref{eqn:hypergeometric_formulaM=1andN=2version2}, we use the identity
$$
\sum_{k=0}^{n}  (-1)^k \binom{n}{k} B_kB_{2,n-k} = nB_{2,n-1}+ \sum_{k=0}^{n} \binom{n}{k} B_kB_{2,n-k}
$$
which holds since $B_1=-1/2$ and $B_{2k+1}=0$ for $k\geq 0$.
\end{proof}

    Observe that formula \eqref{eqn:hypergeometric_formulaM=1andN=2version2} generalizes \eqref{eq:euler} and allows us to calculate the hypergeometric Bernoulli numbers $B_{2,n}$ in terms of the classical Bernoulli numbers $B_n$.  This is useful since the odd $B_{2n-1}$ are known to vanish except for $B_1$ and thus \eqref{eqn:hypergeometric_formulaM=1andN=2version2} allows us to calculate $B_{2,n}$ more efficiently.

Next, we discuss another interesting case of Theorem \ref{th:hypergeometric-bernoulli}, namely when $M=N$.  This is the content of the following corollary.

\begin{corollary} \label{th:hypergeometric-bernoulli-M=N}
    For any positive integer $N$ we have
    \begin{multline}     \label{eqn:hypergeometric_formula-M=N}
\sum_{\substack {0\leq i \leq N-1 \\ 0 \leq j \leq N-1 \\
(i,j)\neq (0,0)}} \left[\sum_{k=0}^n (-1)^k \binom{n}{j;i;k-i} B_{N,n-k-j} B_{N,k-i}\right] \\
=        (-1)^{N-1} \left[\binom{2N}{N}\delta_{n-2N}+(1+(-1)^n)\sum_{j=0}^{N-1}  \binom{n}{N; j}B_{N,n-N-j}\right]
    \end{multline}
\end{corollary}

\noindent Observe that if $n$ is odd, then both sides of \eqref{eqn:hypergeometric_formula-M=N} vanish.  This is clear for the right-hand side because of the delta term and the sign alternation.  As for the left-hand side, this follows from the fact that the $k$-th term of the inner summation corresponding to $(i,j)$ cancels with the $(n-k)$-th term corresponding to $(j,i)$.

In particular, setting $N=2$ in Corollary \ref{th:hypergeometric-bernoulli-M=N} yields

    \begin{corollary} \label{cor:N=M=2} For any even integer $n\geq 0$, we have
        \begin{equation}
        \label{eq:N=M=2}
        \sum_{k=0}^{n} (-1)^k\binom{n}{k}B_{2,n-k}(x)B_{2,k}(x) =\frac{1}{2}\delta_{n-2}+nB_{2,n-1}+B_{2,n}
        \end{equation}
    \end{corollary}
    
\begin{proof}
Assume $n$ is even.  Then substituting $N=2$ into \eqref{eqn:hypergeometric_formula-M=N} yields

\begin{dmath*}
\sum_{k=0}^n (-1)^k \binom{n}{1;0;k} B_{2,n-k-1} B_{2,k} +\sum_{k=0}^n (-1)^k \binom{n}{0;1;k-1} B_{2,n-k} B_{2,k-1} + \sum_{k=0}^n (-1)^k \binom{n}{1;1;k-1} B_{2,n-k-1} B_{2,k-1} = - \binom{4}{2}\delta_{n-4}-2\sum_{j=0}^{1}  \binom{n}{2; j}B_{2,n-2-j}
\end{dmath*}
which simplifies to
\begin{dmath} \label{eq:convolution-M=N=2}
2n\sum_{k=0}^{n-1} (-1)^k \binom{n-1}{k} B_{2,n-k-1} B_{2,k} +n(n-1)\sum_{k=1}^{n-1} (-1)^k \binom{n-2}{k-1} B_{2,n-k-1} B_{2,k-1} =-6\delta_{n-4}-2\sum_{j=0}^{1}  \binom{n}{2; j}B_{2,n-2-j}
\end{dmath}
where we have used the fact that $B_{2,k}=0$ for $k<0$.  Next, observe that the first summation in \eqref{eq:convolution-M=N=2} is an alternating convolution, which vanishes since $n-1$ is odd.  The second summation can also be rewritten as an alternating convolution after re-indexing $k$.  This leads to
$$
n(n-1)\sum_{k=0}^{n-2} (-1)^k \binom{n-2}{k} B_{2,n-k-2} B_{2,k} =6\delta_{n-4}+n(n-1)\sum_{j=0}^{1}  \binom{n-2}{j}B_{2,n-2-j}
$$
and gives a formula for $c_n$ if we re-index $n$:
$$
c_n=\sum_{k=0}^{n} (-1)^k \binom{n}{k} B_{2,n-k} B_{2,k} =\frac{6}{(n+1)(n+2)}\delta_{n-2}+\sum_{j=0}^{1}  \binom{n}{ j}B_{2,n-j}
$$
This completes the proof since $6\delta_{n-2}/[(n+1)(n+2)] = \delta_{n-2}/2$.
\end{proof}

As a corollary, we now derive identities for sums of products involving only the even $B_{2,2n}$ and separately for the odd $B_{2,2n-1}$.   Recall the following convolution formula due to Kamano \cite{kamano}:
    
\begin{theorem}[\cite{kamano}] Let $N$ be a positive integer.  Then
    \begin{equation}
    \label{eq:kamano}
    \sum_{k=0}^{n} \binom{n}{k}B_{N,n-k}B_{N,k} =-\frac{1}{N}[nB_{N,n-1}+(n-N)B_{N,n}]
    \end{equation}
\end{theorem}

It follows from averaging equations \eqref{eq:N=M=2} and  \eqref{eq:kamano} above with $N=2$ that
$$
    \sum_{k=0 \ \mathrm{mod} \ 2}^{n} \binom{n}{k}B_{2,n-k}B_{2,k} =\frac{1}{4}[\delta_{n-2}+nB_{2,n-1}-(n-4)B_{2,n}]
$$
or equivalently,
\begin{corollary} \label{cor:N=M=2-even}
    \begin{equation} \label{eq:N=M=2-even}
    \sum_{k=0}^{n} \binom{2n}{2k}B_{2,2n-2k}B_{2,2k} =\frac{1}{4}[\delta_{2n-2}+2nB_{2,2n-1}-(2n-4)B_{2,2n}]
    \end{equation}
\end{corollary}
\noindent Observe that \eqref{eq:N=M=2-even} generalizes Euler's quadratic formula for the classical Bernoulli numbers:
    \[\sum\limits_{k = 0}^n \binom{2n}{2k} B_{2n - 2k} B_{2k}  = -(2n - 1) B_{2n}\]
    
    On the other hand, subtracting the same two equations above yields
$$
    \sum_{k=1 \ \mathrm{mod} \ 2}^{n} \binom{n}{k}B_{2,n-k}B_{2,k} =-\frac{1}{4}[\delta_{n-2}+3nB_{2,n-1}+nB_{2,n}]
$$
    or equivalently,
    \begin{corollary} \label{cor:N=M=2-odd}
    \begin{equation}
    \sum_{k=1}^{n} \binom{2n}{2k-1}B_{2,2n-2k+1}B_{2,2k-1} =-\frac{1}{4}[\delta_{2n-2}+6nB_{2,2n-1}+2nB_{2,2n}]
    \end{equation}
    \end{corollary}

    \section{Exponential Partial Fraction Expansion}
    In this section we demonstrate that some of the identities obtained in the previous section can also be obtained from exponential generating functions defined for hypergeometric Bernoulli polynomials, which we recall from \eqref{eq:egf-hyper}:
$$
   F_N(x,t)\equiv \frac{(t^N/N!)e^{xt}}{e^t-T_{N-1}(t)}=\sum_{n=0}^{\infty}B_{N,n}(x)\frac{t^n}{n!}
$$    
This was achieved by taking advantage of elementary expansion formulas that hold for special partial fractions, in particular those that we refer to as {\em exponential partials fractions}.  We write out these formulas in the following lemma without proof.

\begin{lemma} \label{le:partial-fraction}   Let $A$ and $B$ be two quantities.  Then
\begin{flalign}
& \mathrm{(1)} \ \frac{1-AB}{(e^t-A)(e^{-t}-B)}=1+\frac{A}{e^t-A}+\frac{B}{e^{-t}-B} & \\
& \mathrm{(2)} \
    \frac{A-B}{(e^t-A)(e^{t}-B)}=\frac{1}{e^t-A}-\frac{1}{e^{t}-B} \label{eq:partia-fraction-2}
\end{flalign}
\end{lemma}

Next, we express Lemma \ref{le:partial-fraction} in terms of the exponential generating function $F_N(x,t)$.

\begin{theorem} \label{th:HBP-partial-fraction} Let $M$ and $N$ be positive integers.
Then
\begin{flalign}
& \mathrm{(1)} \ [1-T_{N-1}(t)T_{M-1}(-t)]F_N(x,t)F_M(x,-t) & \notag \\
& \ \ \ \ \ \  =(-1)^M\frac{t^{M+N}}{M!N!}+(-1)^M\frac{t^M}{M!}T_{N-1}(t)F_N(0,t)+\frac{t^N}{N!}T_{M-1}(-t)F_M(0,-t) & \label{eq:HBP-partial-fraction} \\
& \mathrm{(2)} \ [T_{N-1}(t)-T_{M-1}(t)]F_N(x_1,t)F_M(x_2,t)=\frac{t^M}{M!}F_N(x,t)-\frac{t^N}{N!}F_M(x,t) &  \label{eq:HBP-partial-fraction2}
\end{flalign}

where $x=x_1+x_2$.
\end{theorem}

\begin{proof}
For (1), set $A=T_{N-1}(t)$ and $B=T_{M-1}(-t)$.  Then
\begin{align}
& [1-T_{N-1}(t)T_{M-1}(-t)]F_N(x,t)F_M(x,-t) \notag \\
& \hspace{50pt}  = (-1)^M (t^N/N!)(t^M/M!)\frac{1-AB}{(e^t-A)(e^{-t}-B)} \notag \\
& \hspace{50pt}  = (-1)^M \frac{t^{M+N}}{M!N!} \left(1+\frac{A}{e^t-A}+\frac{B}{e^{-t}-B} \right) \\
& \hspace{50pt} =(-1)^M\frac{t^{M+N}}{M!N!}+(-1)^M\frac{t^M}{M!}T_{N-1}(t)F_N(0,t)+\frac{t^N}{N!}T_{M-1}(-t)F_M(0,-t)
\end{align}
For (2), set $A=T_{N-1}(t)$ and $B=T_{M-1}(t)$.  Then
\begin{align}
[T_{N-1}(t)-T_{M-1}(t)]F_N(x_1,t)F_M(x_2,t) 
& = (t^N/N!)(t^M/M!)e^{(x_1+x_2)t}\frac{A-B}{(e^t-A)(e^{t}-B)} \notag \\
& = \frac{t^{M+N}}{M!N!}e^{(x_1+x_2)t} \left(\frac{1}{e^t-A}-\frac{1}{e^{t}-B} \right) \\
& =\frac{t^M}{M!}F_N(x,t)-\frac{t^N}{N!}F_M(x,t) 
\end{align}
\end{proof}

Define $a_m$ to be coefficients of the polynomial
$$1-T_{N-1}(t)T_{M-1}(-t)= \sum_{m=1}^{M+N-2}a_mt^m$$
We are ready to prove our first main result in this section.

\begin{theorem}  \label{th:HBP-egf} For positive integers $M$ and $N$, we have
\begin{flalign} 
& \mathrm{(1)} \ \sum_{m=1}^{M+N-2} m! a_m \sum_{k=0}^{n-m} (-1)^k \binom{n}{m; k} B_{N,n-m-k}(x)B_{M,k}(x)
 = (-1)^M\binom{M+N}{M}\delta_{n-M-N} & \notag  \\ 
& \hspace{50pt} + (-1)^M\sum_{m=0}^{N-1} \binom{n}{m; M} B_{N,n-m-M} +
(-1)^{n-N}\sum_{m=0}^{M-1}  \binom{n}{m; N}  B_{M,n-m-N} & \label{eq:HBP-egf}
\end{flalign}
and if $N\geq M$,
\begin{flalign}
& \mathrm{(2)} \ \sum_{m=M}^{N-1} \sum_{k=0}^{n-m}  \binom{n}{m; k} B_{N,n-m-k}(x_1)B_{M,k}(x_2) 
= \binom{n}{M} B_{N,n-M}(x) - \binom{n}{N}B_{M,n-N}(x) &  \label{eq:HBP-egf2}
\end{flalign}

\end{theorem}

\begin{proof}
To prove \eqref{eq:HBP-egf}, we begin with the left-hand side of \eqref{eq:HBP-partial-fraction} which simplifies to
\begin{align}
& [1-T_{N-1}(t)T_{M-1}(-t)]F_N(x,t)F_M(x,-t) \\
& = \left(\sum_{m=1}^{M+N-2}a_mt^m\right) \sum_{n=0}^{\infty} \left( \sum_{k=0}^n (-1)^k \binom{n}{k} B_{N,n-k}(x)B_{M,k}(x)\right) \frac{t^n}{n!} \notag \\
& =  \sum_{m=1}^{M+N-2} a_m \sum_{n=0}^{\infty} \sum_{k=0}^n (-1)^k \binom{n}{k} B_{N,n-k}(x)B_{M,k}(x) \frac{t^{n+m}}{n!} \notag \\
& = \sum_{n=0}^{\infty}\left( \sum_{m=1}^{M+N-2} m! a_m \sum_{k=0}^{n-m} (-1)^k \binom{n}{m; k} B_{N,n-m-k}(x)B_{M,k}(x) \right) \frac{t^{n}}{n!}   \label{eq:HBP-egf-result1} 
\end{align}
On the other, the right-hand side of \eqref{eq:HBP-partial-fraction} simplifies to
\begin{align}
&(-1)^M\frac{t^{M+N}}{M!N!}+(-1)^M\frac{t^M}{M!}T_{N-1}(t)F_N(0,t)+\frac{t^N}{N!}T_{M-1}(-t)F_M(0,-t) \notag \\
& =  \sum_{n=0}^{\infty}(-1)^M\binom{M+N}{M}\delta_{n-M-N}\frac{t^n}{n!}+T_{N-1}(t)\sum_{n=0}^{\infty}(-1)^M \binom{n+M}{M} B_{N,n}\frac{t^{n+M}}{(n+M)!} \notag \\
& \ \ \ \ \ +T_{M-1}(-t)\sum_{n=0}^{\infty}(-1)^n\binom{n+N}{N}B_{M,n}\frac{t^{n+N}}{(n+N)!} \notag \\
& =  \sum_{n=0}^{\infty}(-1)^M\binom{M+N}{M}\delta_{n-M-N}\frac{t^n}{n!}+\left(\sum_{m=0}^{N-1}\frac{t^m}{m!}\right)\sum_{n=M}^{\infty}(-1)^M \binom{n}{M} B_{N,n-M}\frac{t^{n}}{n!} \notag \\
& \ \ \ \ \ +\left(\sum_{m=0}^{M-1}(-1)^m\frac{t^m}{m!}\right)\sum_{n=N}^{\infty}(-1)^{n-N}\binom{n}{N}B_{M,n-N}\frac{t^{n}}{n!} \label{eq:HBP-partial-fraction-step1}
\end{align}
Now use the fact that $B_{N,n}=0$ for $n<0$ to rewrite \eqref{eq:HBP-partial-fraction-step1} as
\begin{align}
&(-1)^M\frac{t^{M+N}}{M!N!}+(-1)^M\frac{t^M}{M!}T_{N-1}(t)F_N(0,t)+\frac{t^N}{N!}T_{M-1}(-t)F_M(0,-t) \notag \\
& =  \sum_{n=0}^{\infty}(-1)^M\binom{M+N}{M}\delta_{n-M-N}\frac{t^n}{n!}+\sum_{n=M}^{\infty}\left[\sum_{m=0}^{N-1} (-1)^M \binom{n+N-1}{n+m} \binom{n+m}{M} B_{N,n+m-M}\right]\frac{t^{n+N-1}}{(n+N-1)!} \notag \\
& \ \ \ \ \ +\sum_{n=N}^{\infty} \left[\sum_{m=0}^{M-1} (-1)^{n+M-1-N} \binom{n+M-1}{n+m} \binom{n+m}{N} B_{M,n+m-N}\right]\frac{t^{n+M-1}}{(n+M-1)!} \notag \\
& =  \sum_{n=0}^{\infty}(-1)^M\binom{M+N}{M}\delta_{n-M-N}\frac{t^n}{n!}+\sum_{n=M+N-1}^{\infty}\left[\sum_{m=0}^{N-1} (-1)^M \binom{n}{N-m-1;M} B_{N,n-N+1+m-M}\right]\frac{t^{n}}{n!} \notag \\
& \ \ \ \ \ +\sum_{n=M+N-1}^{\infty} \left[\sum_{m=0}^{M-1} (-1)^{n-N}  \binom{n}{M-m-1;N}  B_{M,n-M+1+m-N}\right]\frac{t^{n}}{n!}
 \label{eq:HBP-egf-result2}
\end{align}
Then re-index the last two summations in \eqref{eq:HBP-egf-result2} and equate it with \eqref{eq:HBP-egf-result1} to obtain \eqref{eq:HBP-egf}  as desired. 

We next prove \eqref{eq:HBP-egf2} similarly. The left-hand side of \eqref{eq:HBP-partial-fraction2} simplifies to
\begin{align}
& [T_{N-1}(t)-T_{M-1}(t)]F_N(x_1,t)F_M(x_2,t) \\
& \hspace{50pt} = \left(\sum_{m=M}^{N-1}\frac{t^m}{m!}\right) \sum_{n=0}^{\infty} \left( \sum_{k=0}^n \binom{n}{k} B_{N,n-k}(x_1)B_{M,k}(x_2)\right) \frac{t^n}{n!} \notag \\
& \hspace{50pt} =  \sum_{m=M}^{N-1} \frac{1}{m!} \sum_{n=0}^{\infty} \sum_{k=0}^n \binom{n}{k} B_{N,n-k}(x_1)B_{M,k}(x_2) \frac{t^{n+m}}{n!} \notag \\
& \hspace{50pt} = \sum_{n=0}^{\infty}\left( \sum_{m=M}^{N-1} \sum_{k=0}^{n-m} \binom{n}{m; k} B_{N,n-m-k}(x_1)B_{M,k}(x_2) \right) \frac{t^{n}}{n!}   \label{eq:HBP-egf2-result1} 
\end{align}
As for the right-hand side of \eqref{eq:HBP-partial-fraction2}, we have
\begin{align}
& \frac{t^M}{M!}F_N(x,t)-\frac{t^N}{N!}F_M(x,t) \notag \\
& \hspace{50pt} = \sum_{n=0}^{\infty} B_{N,n}(x)\frac{t^{n+M}}{M! n!} - \sum_{n=0}^{\infty}B_{M,n}(x)\frac{t^{n+N}}{N!n!} \notag \\
&
\hspace{50pt} =  \sum_{n=M}^{N-1} \binom{n}{M} B_{N,n-M}(x)\frac{t^{n}}{n!} +\sum_{n=N}^{\infty}\left[ \binom{n}{M} B_{N,n-M}(x) - \binom{n}{N}B_{M,n-N}(x)\right]\frac{t^{n}}{n!} 
 \label{eq:HBP-egf2-result2}
\end{align}
Equation \eqref{eq:HBP-egf2} now follows by assuming $N \geq M$ and equating \eqref{eq:HBP-egf2-result1} with \eqref{eq:HBP-egf2-result2}. 
\end{proof}

We note that part (1) of Theorem \ref{th:HBP-egf} is equivalent to Theorem  \ref{eq:hypergeometric-bernoulli} since the right-hand sides of \eqref{eq:hypergeometric-bernoulli} and \eqref{eq:HBP-egf} are identical (up to sign).  For example, if $M=1$ and $N=2$, then it can be shown that \eqref{eq:HBP-egf} and \eqref{eq:HBP-egf2} reduces to \eqref{eqn:hypergeometric_formulaM=1andN=2} and \eqref{eqn:hypergeometric_formulaM=1andN=2version2}, respectively.  If $M=N=2$, then \eqref{eq:HBP-egf} reduces to \eqref{eq:N=M=2}, but \eqref{eq:HBP-egf2} becomes trivial.  This is because both sides of \eqref{eq:HBP-egf2} vanishes when $M=N$.  To obtain non-trivial identities in this case, it is necessary to replace \eqref{eq:partia-fraction-2} with the following formula by applying Leibniz's rule for derivatives:
\begin{align}
\frac{d}{dt}\left[\frac{e^{xt}}{(e^t-A)^p}\right] & = e^{xt}\frac{d}{dt}\left[\frac{1}{(e^t-A)^p}\right] +x\frac{e^{xt}}{(e^t-A)^p} \notag \\
& =-pe^{xt}\frac{e^t-A'}{(e^t-A)^{p+1}} +x\frac{e^{xt}}{(e^t-A)^p} \notag \\
& =p(A'-A)\frac{e^{xt}}{(e^t-A)^{p+1}} +(x-p)\frac{e^{xt}}{(e^t-A)^p} \label{eq:partial-fraction-derivative}
\end{align}

This leads to the following result.  
\begin{lemma} Suppose $z=x_1+x_2+...+x_{p+1}=y_1+y_2+...+y_p$.  Then
\begin{equation} \label{eq:egf-recurrence}
\prod_{k=1}^{p+1}F_N(x_k,t) = \left[\frac{(z-p)t}{pN}+1\right]\prod_{k=1}^{p}F_N(y_k,t)-\left(\frac{t}{pN}\right)\frac{d}{dt}\left[\prod_{k=1}^{p}F_N(y_k,t)\right] \\
\end{equation}
\end{lemma}

\begin{proof}
Set $A=T_{N-1}(t)$.  It follows from \eqref{eq:partial-fraction-derivative} that
\begin{align}
pN\prod_{k=1}^{p+1}F_N(x_k,t) 
& =-t\left(\frac{t^N}{N!}\right)^p p(A'-A)\frac{e^{yt}}{(e^t-A)^{p+1}} \notag \\
& = (z-p)t \frac{(t^N/N!)^p e^{yt}}{(e^t-A)^p} - t \left(\frac{t^N}{N!}\right)^p \frac{d}{dt}\left[\frac{e^{yt}}{(e^t-A)^p}\right] & \notag \\
&=  (z-p)t\prod_{k=1}^{p}F_N(y_k,t) + t \frac{d}{dt}\left[\left(\frac{t^N}{N!}\right)^p\right] \frac{e^{yt}}{(e^t-A)^p}  - t \frac{d}{dt} \left[\frac{(t^N/N!)^p e^{yt}}{(e^t-A)^p} \right] \notag \\
&=  (z-p)t\prod_{k=1}^{p}F_N(y_k,t) + pN\prod_{k=1}^{p}F_N(y_k,t) -t\frac{d}{dt}\left[\prod_{k=1}^{p}F_N(y_k,t)\right] \notag
\end{align}
which proves \eqref{eq:egf-recurrence} as desired.
\end{proof}

Following Dilcher \cite{dilcher} and Kamano \cite{kamano}, we define sums of products of hypergeometric Bernoulli polynomials of order $p$ by
\begin{equation}
S_{N,n}^{(p)}(\hat{x}_p)=\sum_{i_1+i_2+...+i_p=n} \frac{n!}{i_1! i_2!\cdots i_p!} B_{N,i_1}(x_1)B_{N,i_2}(x_2)\cdots B_{N,i_p}(x_p)
\end{equation}
where $\hat{x}_p=(x_1,x_2,...,x_p)$.  Moreover, we set $S_{N,n}^{(p)}(\hat{x}_p)=0$ for $n<0$.  Then observe that
\begin{equation} \label{eq:sums-of-product-LHS-expansion}
\prod_{k=1}^{p+1}F_N(x_k,t) = \sum_{n=0}^{\infty} S_{N,n}^{(p+1)}(\hat{x}_{p+1})\frac{t^n}{n!}
\end{equation}
On the other hand, we can expand the right-hand side of \eqref{eq:egf-recurrence} as follows by first defining
$$\hat{y}_p =(y_1,y_2,...,y_p)$$
so that $z=x_1+x_2+...+x_{p+1}=y_1+y_2+...+y_p$.  It follows that
\begin{flalign}
&\hspace{30pt} \left[\frac{(z-p)t}{pN}+1\right]\prod_{k=1}^{p}F_N(y_k,t)-\left(\frac{t}{pN}\right)\frac{d}{dt}\left[\prod_{k=1}^{p}F_N(y_k,t)\right]  & \notag \\
&  \hspace{60pt} = \left[\frac{(z-p)t}{pN}+1\right]  \sum_{n=0}^{\infty} \left[S_{N,n}^{(p)}(\hat{y}_p)\right]\frac{t^n}{n!} -\sum_{n=0}^{\infty}\left[\frac{n}{pN}S_{N,n}^{(p)}(\hat{y}_p)\right]\frac{t^n}{n!} & \notag \\
& \hspace{60pt} =\sum_{n=0}^{\infty} \left[\frac{(y-p)}{pN}S_{N,n}^{(p)}(\hat{y}_p)\right]\frac{t^{n+1}}{n!} + \sum_{n=0}^{\infty} \left[S_{N,n}^{(p)}(\hat{y}_p)\right]\frac{t^n}{n!} -\sum_{n=0}^{\infty} \left[\frac{n}{pN}S_{N,n}^{(p)}(\hat{y}_p)\right]\frac{t^n}{n!} & \notag \\
& \hspace{60pt} = \sum_{n=0}^{\infty} \left[\frac{n(y-p)}{pN}S_{N,n-1}^{(p)}(\hat{y}_p)+ S_{N,n}^{(p)}(\hat{y}_p)+ \frac{n}{pN}S_{N,n}^{(p)}(\hat{y}_p)\right]\frac{t^{n+1}}{n!} & \label{eq:sums-of-product-RHS-expansion}
\end{flalign}

By equating coefficients of \eqref{eq:sums-of-product-LHS-expansion} and \eqref{eq:sums-of-product-RHS-expansion} we obtain the following theorem, which establishes a recurrence for $S_{N,n}^{(p)}(\hat{x}_p)$ and generalizes Kamano's recurrence formula for sums of products of hypergeometric Bernoulli numbers (expressed in terms of a modified hypergeometric zeta function) given in \cite{kamano}  (Lemma 3.1).

\begin{theorem} \label{th:sum-product-recurrence} Let $\hat{x}_{p+1}=(x_1,x_2,...,x_{p+1})$ and $\hat{y}_p=(y_1,y_2,...,y_p)$ be such that
$$
z=x_1+x_2+...+x_{p+1}=y_1+y_2+...+y_p
$$
Then
\begin{equation} \label{eq:sum-product-recurrence}
S_{N,n}^{(p+1)}(\hat{x}_{p+1})=\frac{1}{pN}\left[(pN-n)S_{N,n}^{(p)}(\hat{y}_p)+n(z-p)S_{N,n-1}^{(p)}(\hat{y}_p)\right]
\end{equation}
\end{theorem}

As an example, we use Theorem \ref{th:sum-product-recurrence} to recursively generate formulas for $S_{N,n}^{(p)}(\hat{x}_p)$ for $p$ equals 2 and 3 in terms of $S_{N,n}^{(1)}(z)=B_{N,n}(z)$.  These formulas generalize those given by Dilcher \cite{dilcher} (equations (3.2) and (3.8)) for Bernoulli polynomials and by Kamano \cite{kamano} (conclusion section) for hypergeometric Bernoulli numbers.

\begin{example}
\begin{flalign}
\mathrm{(1)} \  S_{N,n}^{(2)}(\hat{x}_2) & =\frac{1}{N}\left[(N-n)B_{N,n}(z)+n(z-1)B_{N,n-1}(z)\right]  \\
\mathrm{(2)} \ S_{N,n}^{(3)}(\hat{x}_3)  & = \frac{1}{2N^2}[(N-n)(2N-n)B_{N,n}(z) +((2N-n)n(z-1)+n(z-2)(N-n+1))B_{N,n-1}(z)  \notag \\
& \ \ \ \ \ \ \ \ \ \ +n(n-1)(z-1)(z-2)B_{N,n-2}(z)]
\end{flalign}
\end{example}

To obtain a general formula for $S_{N,n}^{(p)}(\hat{x}_p)$ as a linear combination of hypergeometric Bernoulli polynomials, we take advantage of a theorem proven by the first author in \cite{nguyen} regarding two-dimensional sequences satisfying recurrences of the type similar to \eqref{eq:sum-product-recurrence}.  
Towards this end, suppose a two-dimensional sequence $x(n,k)$ satisfies the generalized triangular recurrence
\begin{equation} \label{eq:triangular-recurrence}
\begin{array}{l} 
{x(n,k)=a(n,k)x(n-1,k)+b(n,k)x(n-1,k-1)} \\ {x(0,k)=f(k)} 
\end{array}
\end{equation}
where $a(n,k)$ and $b(n,k)$ are known two-dimensional sequences, $k$ is an integer, and $n$ is a non-negative integer.  Next, denote by $A_{n}(m)$ to be the set of $m$-element subsets of $A=\{1,2,...,n\}$  and let $\sigma =\{ i_{1} ,i_{2} ,...,i_{m} \} \in A_{n} (m)$ be such a subset.  We define the \textit{rank} of a positive integer $j$ relative to $\sigma $, and denote it by $R_{\sigma } (j)$, to be the number of elements in $\sigma $ that are greater than $j$, i.e.
\[R_{\sigma } (j)=\left|\{ i\in \sigma :i>j\} \right|\]
Then define the product
\begin{equation} \label{eq:recurrrence-product}
\pi_{a,b}(\bar{\sigma},\sigma;k)=\prod_{s=1}^{n-m}a(j_s,k-R_{\sigma}(j_s)) \prod_{r=1}^{m}b(i_r,k-R_{\sigma}(i_r))
\end{equation}
where $\bar{\sigma}=\{j_1,j_2,...,j_{n-m}\}$ denotes the complement of $\sigma$ in $A$.  The following theorem gives a general formula for $x(n,k)$ in terms of $f(k)$.

\begin{theorem}[\cite{nguyen}, Theorem 4]  \label{th:two-dimensional-sequence}
\begin{equation} \label{eq:two-dimensional-sequence} 
x(n,k)=\sum _{m=0}^{n}\left(\sum _{\sigma \in A_{n} (m)}\pi_{a,b}(\bar{\sigma },\sigma ;k) \right)f(k-m).
\end{equation} 
\end{theorem}
We now apply Theorem \eqref{th:two-dimensional-sequence} to obtain our desired formula for $S_{N,n}^{(p)}(\hat{x}_p)$.
\begin{theorem} Let $a(p,n)=1-\frac{n}{pN} $, $b(p,n)=\frac{n}{N}\left(\frac{z}{p} -1\right)$, and $f(n)=B_{N,n}(z)$.  Define $\pi_{a,b}(\bar{\sigma },\sigma ;n)$ as in \eqref{eq:recurrrence-product}.  Then
\begin{equation} \label{eq:sums-of-product-direct-formiula} 
S_{N,n}^{(p)} (\hat{x}_p)=\sum _{k=0}^{p-1}\left(\sum _{\sigma \in A_{p-1} (k)}\pi_{a,b}(\bar{\sigma },\sigma ;n) \right)B_{N,n-k} ( z) .
\end{equation} 
\end{theorem}

\begin{proof}
Set $x(p,n)=S_{N,n}^{(p+1)} (\hat{x}_p)$.  Since $x(n,p)$ satisfies the recurrence \eqref{eq:triangular-recurrence}, equation \eqref{eq:sums-of-product-direct-formiula} follows immediately from Theorem \ref{th:two-dimensional-sequence}.
\end{proof}

\section{Conclusion}
In this paper we established new convolution identities for hypergeometric Bernoulli polynomials by considering two different, but equivalent, definitions of Appell sequences.  Some of these identities generalize those found by Euler and Dilcher \cite{dilcher} for classical Bernoulli numbers and polynomials and by Kamano \cite{kamano} for hypergeometric Bernoulli numbers. 

We end by noting that the expansion formulas in Lemma \ref{le:partial-fraction} can be generalized to contain more than two factors, e.g.,
\begin{equation}\label{eq:partial-fraction-3factors}
\frac{(A-B)(A-C)(B-C)}{(e^t-A)(e^t-B)(e^t-C)} = \frac{B-C}{e^t-A} - \frac{A-C}{e^t-B} + \frac{A-B}{e^t-C}
\end{equation}
where $A$, $B$, and $C$ are three different quantities.  It is then natural to use \eqref{eq:partial-fraction-3factors} to derive higher-order identities, thereby generalizing Theorem \ref{th:HBP-egf}.

\

\noindent {\em Acknowlegements.} The authors would like to thank Hunduma Legesse Geleta (Addis Ababa University, Ethiopia) for his helpful comments and corrections.

\end{document}